\documentclass[a4paper]{amsart}

\usepackage{amssymb}
\usepackage{amsmath}
\usepackage{amsthm}
\usepackage{booktabs}
\usepackage{cite} 
\usepackage{color}
\usepackage{enumitem}
\usepackage{graphicx}
\usepackage{latexsym}
\usepackage{pxfonts}
\usepackage{url}
\usepackage[all]{xy}

\theoremstyle{definition}
\newtheorem{definition}{Definition}[section]

\newtheorem{example}[definition]{Example}

\newtheorem{remark}[definition]{Remark}

\theoremstyle{plain}
\newtheorem{lemma}[definition]{Lemma}
\newtheorem{proposition}[definition]{Proposition}
\newtheorem{theorem}[definition]{Theorem}
\newtheorem{corollary}[definition]{Corollary}
\newtheorem{conjecture}[definition]{Conjecture}
\newtheorem*{itheorem}{Theorem}

\DeclareMathOperator{\Rect}{\mathcal{R}\mathit{ect}}
\DeclareMathOperator{\End}{End}

\DeclareMathOperator{\Map}{Map}
\DeclareMathOperator{\id}{id}

\DeclareMathOperator{\Fin}{\mathbf{Fin}}
\DeclareMathOperator{\Vect}{\mathbf{Vect}}
\DeclareMathOperator{\Ch}{\mathbf{Ch}}
\DeclareMathOperator{\Set}{\mathbf{Set}}

\DeclareMathOperator{\Ass}{\mathcal{A}\mathit{ss}}
\DeclareMathOperator{\uAss}{\mathit{u}\mathcal{A}\mathit{ss}}
\DeclareMathOperator{\Com}{\mathcal{C}\mathit{om}}
\DeclareMathOperator{\uCom}{\mathit{u}\mathcal{C}\mathit{om}}

\renewcommand{\Box}{\mathbin{\square}}

\DeclareMathAlphabet{\pazocal}{OMS}{zplm}{m}{n}

\def\calC{\pazocal{C}}

\def\calF{\pazocal{F}}
\def\calG{\pazocal{G}}
\def\calH{\pazocal{H}}
\def\calI{\pazocal{I}}

\def\calM{\pazocal{M}}

\def\calO{\pazocal{O}}
\def\calP{\pazocal{P}}
\def\calQ{\pazocal{Q}}
\def\calR{\pazocal{R}}
\def\calS{\pazocal{S}}
\def\calT{\pazocal{T}}

\def\calV{\pazocal{V}}
\def\calW{\pazocal{W}}
\def\calX{\pazocal{X}}
\def\calY{\pazocal{Y}}
\def\calZ{\pazocal{Z}}

\DeclareMathAlphabet{\mathbbold}{U}{bbold}{m}{n}

\def\k{\mathbbold{k}}

\DeclareMathOperator\Smod{\mathbb{S}-mod}

\newcommand{\ac}{\scriptstyle \text{\rm !`}}

\makeatletter
\def\l@section{\@tocline{1}{0pt}{1pc}{}{}}
\def\l@subsection{\@tocline{2}{0pt}{1pc}{4.6em}{}}
\def\l@subsubsection{\@tocline{3}{0pt}{1pc}{7.6em}{}}
\renewcommand{\tocsection}[3]{%
  \indentlabel{\@ifnotempty{#2}{\makebox[1.25em][l]{\ignorespaces#1#2.}}}#3}
\renewcommand{\tocsubsection}[3]{%
  \indentlabel{\@ifnotempty{#2}{\hspace*{1.25em}\makebox[2.00em][l]{\ignorespaces#1#2.}}}#3}
\renewcommand{\tocsubsubsection}[3]{%
  \indentlabel{\@ifnotempty{#2}{\hspace*{3.25em}\makebox[2.75em][l]{\ignorespaces#1#2.}}}#3}
\makeatother

\allowdisplaybreaks

\begin{document}

\title[Boardman--Vogt tensor products of absolutely free operads]{Boardman--Vogt tensor products\\ of absolutely free operads}

\author{Murray Bremner}

\address{Department of Mathematics and Statistics, University of Saskatchewan, Saskatoon, Canada}

\email{bremner@math.usask.ca}

\author{Vladimir Dotsenko}

\address{School of Mathematics, Trinity College Dublin, Ireland, and Departamento de Matem\'aticas, CINVESTAV-IPN,  Col. San Pedro Zacatenco, M\'exico, D.F., CP 07360, Mexico}

\email{vdots@maths.tcd.ie}

\thanks{The research of Murray Bremner was supported by a Discovery Grant from NSERC, 
the Natural Sciences and Engineering Research Council of Canada.}

\dedicatory{To the memory of Trevor Evans (1925--1991), the pioneer of interchange laws in universal algebra}

\subjclass[2010]{Primary 
18D50. 
Secondary 
05A15,
05E15,
18D05, 
18G10,
52C22}

\keywords{absolutely free operad, algebraic operad,
interchange law,
nonassociative algebra,
Boardman-Vogt tensor product of symmetric operads, operad of little $d$-rectangles,
rectangular partitions of the unit $d$-dimensional cube}

\begin{abstract}
We establish a combinatorial model for the Boardman--Vogt tensor product 
of several absolutely free operads, that is free symmetric operads that are also free as  
$\mathbb{S}$-modules. Our results imply that such a tensor product is always a
free $\mathbb{S}$-module, in contrast with the results of Kock and Bremner--Madariaga
on hidden commutativity for the Boardman--Vogt tensor square of the operad of non-unital 
associative algebras. 
\end{abstract}

\maketitle

\section*{Introduction}

\subsection*{Interchange law}
Consider two binary operations $\rightarrow$ and $\uparrow$ on the same set $X$.
These operations are said to satisfy the \emph{interchange law} if (for all $x_1,\ldots, x_4 \in X$) 
\begin{equation}
\label{eq:interchangelaw}
( x_1 \uparrow x_2 ) \rightarrow ( x_3 \uparrow x_4 ) = ( x_1 \rightarrow x_3 ) \uparrow ( x_2 \rightarrow x_4 ).
\end{equation}
Note that this relation does not require
$X$ to possess any extra structure, e.g. it is not required to be an Abelian group, or a vector space.
Unlike familiar relations like associativity, each term in this relation involves three operation symbols, so this
 is, in the language of algebraic operads, a cubic relation. 

\smallskip

Geometrically, the interchange law expresses 
the equivalence of the two sequences of bisections which partition a square into four equal squares:
  \begin{align*}
  &
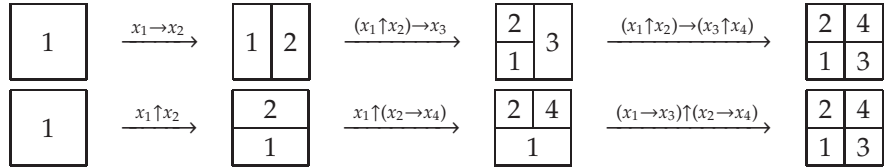
  \begin{xy} 
  (0,-5 ) = "1"; 
  (0,5 ) = "2"; 
  (10,-5 ) = "3"; 
  (10,5 ) = "4"; 
  { \ar@{-} "1"; "2" }; 
  { \ar@{-} "3"; "4" }; 
  { \ar@{-} "1"; "3" }; 
  { \ar@{-} "2"; "4" };
 {(5,0)}*{1}
  \end{xy}
  \quad  \xrightarrow{\; x_1 \rightarrow x_2 \;} \quad
  \begin{xy}
  (0,-5 ) = "1"; 
  (0,5 ) = "2"; 
  (5,-5 ) = "3"; 
  (5,5 ) = "4"; 
  { \ar@{-} "1"; "2" }; 
  { \ar@{-} "3"; "4" }; 
  { \ar@{-} "1"; "3" }; 
  { \ar@{-} "2"; "4" }; 
 {(2.5,0)}*{1};
  (5,-5 ) = "5"; 
  (5,5 ) = "6"; 
  (10,-5 ) = "7"; 
  (10,5 ) = "8"; 
  { \ar@{-} "5"; "6" }; 
  { \ar@{-} "7"; "8" }; 
  { \ar@{-} "5"; "7" }; 
  { \ar@{-} "6"; "8" }; 
 {(7.5,0)}*{2}
  \end{xy}
  \quad  \xrightarrow{\; (x_1\uparrow x_2)\rightarrow x_3 \;} \quad
  \begin{xy}
  (0,-5 ) = "1"; 
  (0,5 ) = "2"; 
  (5,-5 ) = "3"; 
  (5,5 ) = "4"; 
  { \ar@{-} "1"; "2" }; 
  { \ar@{-} "3"; "4" }; 
  { \ar@{-} "1"; "3" }; 
  { \ar@{-} "2"; "4" }; 
  (5,-5 ) = "5"; 
  (5,5 ) = "6"; 
  (10,-5 ) = "7"; 
  (10,5 ) = "8"; 
  { \ar@{-} "5"; "6" }; 
  { \ar@{-} "7"; "8" }; 
  { \ar@{-} "5"; "7" }; 
  { \ar@{-} "6"; "8" }; 
  (0,0 ) = "9";
  (5,0 ) = "10";
  { \ar@{-} "9"; "10" };   
 {(2.5,-2.5)}*{1};
 {(2.5,2.5)}*{2};
 {(7.5,0)}*{3};
  \end{xy}
  \quad \xrightarrow{\,\,\, (x_1\uparrow x_2)\rightarrow (x_3\uparrow x_4) \,\,\,} \quad
  \begin{xy}
  (0,-5 ) = "1"; 
  (0,5 ) = "2"; 
  (5,-5 ) = "3"; 
  (5,5 ) = "4"; 
  { \ar@{-} "1"; "2" }; 
  { \ar@{-} "3"; "4" }; 
  { \ar@{-} "1"; "3" }; 
  { \ar@{-} "2"; "4" }; 
  (5,-5 ) = "5"; 
  (5,5 ) = "6"; 
  (10,-5 ) = "7"; 
  (10,5 ) = "8"; 
  { \ar@{-} "5"; "6" }; 
  { \ar@{-} "7"; "8" }; 
  { \ar@{-} "5"; "7" }; 
  { \ar@{-} "6"; "8" }; 
  (0,0 ) = "9";
  (5,0 ) = "10";
  (10,0 ) = "11";
  { \ar@{-} "9"; "10" };   
  { \ar@{-} "10"; "11" };   
 {(2.5,-2.5)}*{1};
 {(2.5,2.5)}*{2};
 {(7.5,-2.5)}*{3};
 {(7.5,2.5)}*{4};
  \end{xy}
  \\
  &
  \begin{xy} 
  (0,-5 ) = "1"; 
  (0,5) = "2"; 
  (10,-5) = "3"; 
  (10,5) = "4"; 
  { \ar@{-} "1"; "2" }; 
  { \ar@{-} "3"; "4" }; 
  { \ar@{-} "1"; "3" }; 
  { \ar@{-} "2"; "4" };
 {(5,0)}*{1};
  \end{xy}
  \quad \xrightarrow{\,\,\, x_1\uparrow x_2 \,\,\,} \quad
  \begin{xy}
  (0,-5) = "1"; 
  (0,0) = "2"; 
  (10,-5) = "3"; 
  (10,0 ) = "4"; 
  { \ar@{-} "1"; "2" }; 
  { \ar@{-} "3"; "4" }; 
  { \ar@{-} "1"; "3" }; 
  { \ar@{-} "2"; "4" }; 
  (0,0) = "5"; 
  (0,5) = "6"; 
  (10,0) = "7"; 
  (10,5) = "8"; 
  { \ar@{-} "5"; "6" }; 
  { \ar@{-} "7"; "8" }; 
  { \ar@{-} "5"; "7" }; 
  { \ar@{-} "6"; "8" }; 
 {(5,-2.5)}*{1};
 {(5,2.5)}*{2};
  \end{xy}
  \quad  \xrightarrow{\; x_1\uparrow (x_2\rightarrow x_4) \;} \quad
  \begin{xy}
  (0,-5) = "1"; 
  (0,0) = "2"; 
  (10,-5) = "3"; 
  (10,0) = "4"; 
  { \ar@{-} "1"; "2" }; 
  { \ar@{-} "3"; "4" }; 
  { \ar@{-} "1"; "3" }; 
  { \ar@{-} "2"; "4" }; 
  (0,0) = "5"; 
  (0,5) = "6"; 
  (10,0) = "7"; 
  (10,5) = "8"; 
  { \ar@{-} "5"; "6" }; 
  { \ar@{-} "7"; "8" }; 
  { \ar@{-} "5"; "7" }; 
  { \ar@{-} "6"; "8" }; 
  (5,5) = "9"; 
  (5,0) = "10"; 
  { \ar@{-} "9"; "10" };   
 {(5,-2.5)}*{1};
 {(2.5,2.5)}*{2};
 {(7.5,2.5)}*{4};
  \end{xy}
  \quad\xrightarrow{\; (x_1\rightarrow x_3)\uparrow (x_2\rightarrow x_4) \;} \quad
  \begin{xy}
  (0,-5) = "1"; 
  (0,0) = "2"; 
  (10,-5) = "3"; 
  (10,0) = "4"; 
  { \ar@{-} "1"; "2" }; 
  { \ar@{-} "3"; "4" }; 
  { \ar@{-} "1"; "3" }; 
  { \ar@{-} "2"; "4" }; 
  (0,0) = "5"; 
  (0,5) = "6"; 
  (10,0) = "7"; 
  (10,5) = "8"; 
  { \ar@{-} "5"; "6" }; 
  { \ar@{-} "7"; "8" }; 
  { \ar@{-} "5"; "7" }; 
  { \ar@{-} "6"; "8" }; 
  (5,5) = "9"; 
  (5,0) = "10"; 
  (5,-5) = "11"; 
  { \ar@{-} "9"; "10" };   
  { \ar@{-} "10"; "11" };   
 {(2.5,-2.5)}*{1};
 {(2.5,2.5)}*{2};
 {(7.5,-2.5)}*{3};
 {(7.5,2.5)}*{4};
  \end{xy}
  \end{align*}

\smallskip 

An important toy model of an interchange law is that between the two operations on PROPs of endomorphisms. 
Recall that the collection of sets
\[
\End_X =
 \left\{\End_X(p,q)\right\}_{p, q \ge 0} =
 \left\{\Map(X^p, X^q) \right\}_{p, q \ge 0} 
\]
is equipped with two associative operations:
\begin{itemize}[leftmargin=*]
\item[$-$]
The \emph{vertical composition}
from $\End_X(p,q) \times \End_X(q,r)$ to $\End_X(p,r)$
which is the composition of linear maps:
if $f\colon X^p \to X^q$
and $g\colon X^q \to X^r$ then
\[
f \uparrow g \colon X^p \xrightarrow{\; g \circ f \;} X^r.
\]
\item[$-$]
The \emph{horizontal composition}
from $\End_X(p,q) \times \End_X(r,s)$ to $\End_X(p+r,q+s)$
which is induced by the direct product of maps:
if $f \colon X^p \to X^q$ 
and $g \colon X^r \to X^s$ then
\[
f \rightarrow g \colon
X^{p+r} \cong X^p \times X^r
\xrightarrow{\; f \times g \;}
X^q \times X^s \cong X^{q+s}.
\]
\end{itemize}
These two operations are related by 
the interchange law \eqref{eq:interchangelaw}.

\smallskip

Two binary operations satisfying the interchange law seem to have first appeared explicitly in the mathematical literature in Godement's
``five rules of functorial calculus''~\cite[Appendix \S1, equation (V)]{Godement1958}. More generally, one can talk about interchange for operations of arbitrary arities. The corresponding definition appeared, independently, in work of Evans~\cite{Evans1962} and of Boardman and Vogt~\cite{BV1973}. The latter reference has become the definitive 
source on interchange of algebraic structures, encoding it under the name of Boardman--Vogt tensor product of operads;
its influence on algebraic topology and higher category theory is hard to overestimate. By contrast, the former reference 
remained mostly unnoticed (even by Mathematical Reviews). 

\subsection*{Geometry of interchange}

The geometric model of the interchange law for two binary operations that we mentioned above utilises subdivisions of the unit square into several pieces which are obtained by iterated bisections orthogonal to the coordinate axes. This geometric model admits a straightforward generalisation to $d$ dimensions. In this case the combinatorial objects of interest are subdivisions of the unit cube into $d$-dimensional rectangles with disjoint interiors by a sequence of bisections orthogonal to the coordinate axes. The $d$ interchanging binary operations are represented by bisections orthogonal to the $d$ coordinate hyperplanes. 

Let us remark that such subdivisions of the unit cube are subsets of the components of the operad of little $d$-cubes (or, more precisely, little $d$-rectangles), and in fact form a suboperad. However, these subsets are discrete, and therefore exhibit rigidity that renders the connection somewhat superficial; in particular, in the homology of the operad of little $d$-cubes the corresponding operad collapses into $\Com$, the operad of commutative associative algebras. It is also worth noting that the notion of a subdivision we are working with is different from the commonly considered partitions of the unit cube in the combinatorics literature; the closest but still different notion is that of the so called ``guillotine partitions'', or  ``slicing floorplans''; see the recent paper~\cite{ABMP2014}  of Asinowski, Barequet, Mansour and Pinter and references therein. 

The subdivisions of the unit $d$-cube have $d$ interchanging binary products, that is, the action of the Boardman--Vogt tensor product of $d$ copies of the absolutely free operad on one binary generator. There exists a similar geometric model for any $d$-fold Boardman--Vogt tensor product 
 \[
\calT(\calX_1)\otimes\cdots\otimes\calT(\calX_d)
 \] 
of several absolutely free operads. Namely, to encode a generator of arity $m>2$, we may also consider subdivisions into $m$ equal parts using $m-1$ parallel hyperplanes, and to have several generators for the same operad $\calT(\calX_i)$, we may assign to hyperplane cuts of the same directions labels which distinguish them from one another.
This leads to the general notion of $(\calX_1,\ldots,\calX_d)$-subdivisions of the unit cube, and to the structure of an operad on those subdivisions which we call the  \emph{cut operad}. 

\subsection*{Homological methods}

To establish that our geometric model encodes Boardman--Vogt tensor products faithfully (Theorem \ref{th:Iso}), it turns out to be crucial to move from combinatorics to homological algebra. In a way, all key results of this paper are connected through a conceptual result on right module resolutions. To state that result, recall that the category of $\mathbb{S}$-modules has a monoidal structure $\Box$, called the matrix product by Dwyer and Hess in~\cite{DH2014}, or the arithmetic product by Maia and M\'endez in~\cite{MM2008}, which categorifies the product of Dirichlet series. 

\begin{itheorem}[Th.~\ref{maintheorem}]
Let $\calT(\calX_1)$, \ldots, $\calT(\calX_d)$ be reduced connected absolutely free set operads. There exists a minimal resolution 
 \[
\big((\mathbb{I}\oplus\k\calX_1)\Box\cdots\Box(\mathbb{I}\oplus\k\calX_d)\big)\circ\big(\calT(\calX_1)\otimes\cdots\otimes\calT(\calX_d)\big)
 \]
of the augmentation module $\mathbb{I}$ over (the linearised version of) the $d$-fold Boardman--Vogt tensor product $\calT(\calX_1)\otimes\cdots\otimes\calT(\calX_d)$ by free right modules. Here the homological degree of all factors $\mathbb{I}$ is equal to zero, and the homological degree of $\calX_k$ is equal to $1$ for all $1\le k\le d$.
\end{itheorem}

\subsection*{Organisation of the paper} 

The paper is organised as follows. In Section \ref{sec:recoll}, we recall the key relevant definitions of the theory of operads. In Section \ref{sec:subdiv} we create, in three easy steps, a combinatorial set-up for modelling interchange, the $\calX_\bullet$-subdivisions of the unit $d$-cube. In Section \ref{sec:proofs}, we establish that $\calX_\bullet$-subdivisions encode the Boardman--Vogt product faithfully, in other words, that the cut operad determined by the datum $(\calX_1,\ldots,\calX_d)$ is isomorphic to the $d$-fold tensor product $\calT(\calX_1)\otimes\cdots\otimes\calT(\calX_d)$, and prove Theorem \ref{maintheorem} stated above. In Section \ref{sec:conclusion}, we discuss a possible generalisation of that theorem and its limitations.

\subsection*{Acknowledgements } We would like to thank Sara Madariaga for useful discussions at an early stage of work on this project.  The second author is grateful to Kathryn Hess for the reference \cite{DH2014} where the matrix product of collections is related to the Boardman--Vogt tensor product and especially for sharing her work in progress with William Dwyer and Ben Knudsen which gave us an \emph{a posteriori} intuitive explanation of our homological result.

\section{Recollections}\label{sec:recoll}

We refer the reader to the comprehensive monograph \cite{LV2012} by Loday and Vallette for background on algebraic operads, and only recall some of the notions of particular importance for this paper.

We denote by $\Fin$ the category of nonempty finite sets (with bijections as morphisms); we use the ``topologist's notation'' $\underline{n}=\{1,\ldots,n\}$. Underlying objects of all operads of these paper will be objects of one of the following three symmetric monoidal categories: the category $\Set$ of finite sets (with all maps as morphisms), the category $\Vect$ of finite-dimensional vector spaces (with all linear maps as morphisms), or the category $\Ch$ of nonnegatively graded chain complexes with finite-dimensional components (with all chain maps as morphisms). Denote one of those categories by $\mathcal{C}$. Recall that a \emph{($\mathcal{C}$-valued) symmetric collection} (or an \emph{$\mathbb{S}$-module}) is a contravariant functor from the category~$\Fin$ to $\mathcal{C}$. The category $\Smod$ of symmetric collections has symmetric collections as objects, and natural transformations of functors as morphisms. An immediate consequence of functoriality is that for every $\mathbb{S}$-module $\calF$ the object $\calF(\underline{n})$ acquires a right action of $S_n$, the group of automorphisms of $\underline{n}$ (which explains the terminology);  we denote by $v.\sigma$ the result of the action of $\sigma\in S_n$ on $v\in\calF(\underline{n})$.  We say that a symmetric collection $\calM$ is \emph{free} if for each $n$ the action of $S_n$ on $\calM(\underline{n})$ is free. 

\subsection{Composition of symmetric collections}

We begin by recalling one well known monoidal structure on $\Smod$. 

\begin{definition}
Let $\calP$ and $\calQ$ be two symmetric collections. The \emph{(symmetric) composition} $\calP\circ\calQ$ is defined by the formula
 \[
(\calP\circ\calQ)(X):=\bigsqcup_{k}\calP(\underline{k})\bigotimes_{S_k}\left(\bigsqcup_{f\colon 
X\twoheadrightarrow\underline{k}}\calQ(f^{-1}(1))\otimes\ldots\otimes\calQ(f^{-1}(k))\right),
 \]
where the sum is taken over all surjections~$f$.
\end{definition}

Recall that the \emph{unit collection} $\mathbb{I}$ is defined as follows:
 \[
\mathbb{I}(X)=
\begin{cases}
\mathbbold{1}, \quad  |X|=1,\\
0, \quad  |X|\ne 1,\end{cases}
 \]
where $0$ is the initial object of $\mathcal{C}$. It is well known that the operation $\circ$ makes $\Smod$ into a monoidal category with the unit object $\mathbb{I}$. Monoids in $(\Smod, \circ,\mathbb{I})$ are known as \emph{symmetric operads}. The structure map $\calO\circ\calO\to\calO$ is denoted by $\gamma_\calO$, or simply $\gamma$ where there is no ambiguity. A $\k$-linear symmetric operad $\calO$ is said to be \emph{augmented} if it is equipped with a morphism $\epsilon\colon\calO\to\mathbb{I}$ satisfying $\epsilon\eta=\id$, where $\eta\colon\mathbb{I}\to\calO$ is the unit of the monoid~$\calO$.  

Unless otherwise stated, all operads we work with are \emph{reduced} (that is, $\calO(\underline{0})=0$) and \emph{connected} (that is, $\calO(\underline{1})=\mathbbold{1}$). In the linear context, such operads are automatically augmented, with augmentation being the quotient by the ideal of elements of arity greater than one. 

We say that an operad is \emph{absolutely free} if it is generated by elements that possess no symmetries and satisfy no relations. In other words, an absolutely free operad is a free operad generated by a free symmetric collection.

\subsection{Matrix product of symmetric collections}

The next definition we recall here is much less known. It was first proposed by Maia and M\'endez in~\cite{MM2008} under the name ``arithmetic product'' in order to categorify the Dirichlet product of two sequences of numbers, and then rediscovered by Dwyer and Hess in~\cite{DH2014} under the name ``matrix monoidal structure''. We shall keep the latter name because we feel that it serves as a better illustration of the underlying combinatorics.

\begin{definition}
Let $\calX$ and $\calY$ be two symmetric collections. The \emph{matrix product} $\calX\Box\calY$ is defined by the formula
 \[
(\calX\Box\calY)(X):=\bigsqcup_{(\pi,\tau)}\calX(\pi)\otimes\calY(\tau) ,
 \]
where the sum is taken over all pairs of orthogonal set partitions 
 \[
\pi=\{\pi_1,\ldots,\pi_k\}, \quad \tau=\{\tau_1,\ldots,\tau_l\}
 \]
of $X$, so that $X=\pi_1\sqcup \cdots\sqcup \pi_k=\tau_1\sqcup\cdots\sqcup\tau_l$ and $|\pi_i\cap \tau_j|=1$ for all $i=1,\ldots,k$ and $j=1,\ldots,l$. 
\end{definition}

\smallskip 

It is known that the operation $\Box$ makes $\Smod$ into a monoidal category with the unit object $\mathbb{I}$, see~\cite{MM2008}. More amusingly (although not immediately important) for the purpose of this paper, Dwyer and Hess established in \cite[Prop.~1.20]{DH2014} that there exists a natural transformation
 \[
\sigma\colon (\calV\circ\calW)\Box(\calY\circ\calZ)\to(\calV\Box\calY)\circ(\calW\Box\calZ) , 
 \]
so the interchange law manifests itself once again!





\subsection{Boardman--Vogt tensor product of operads} 

The third monoidal structure that we define here is the monoidal structure on the category of symmetric set operads, introduced by Boardman and Vogt in~\cite{BV1973}, and extensively used in algebraic topology since then. Throughout this section, all operads are assumed to be operads in $\Set$. 

\begin{definition}
Let $\calP$ and $\calQ$ be two symmetric operads. The \emph{Boardman--Vogt tensor product} $\calP\otimes\calQ$ is defined by the formula
 \[
\calP\otimes\calQ = (\calP\sqcup\calQ) / \calI ,
 \]
where $\calI$ is the ideal in the coproduct (free product) of $\calP$ and $\calQ$ generated by all elements of $\calP\sqcup\calQ$ of the form
\begin{equation}\label{eq:IC}
\gamma(p;q,\ldots,q)-\gamma(q;p,\ldots,p).\sigma_{k,l} ,
\end{equation}
where $p\in\calP(\underline{k})$, and $q\in\calQ(\underline{l})$, and $\sigma\in S_{kl}$ which ``exchanges rows and columns'',
that is for each $1\le (i-1)l+j\le kl$ with $1\le i\le k$ and $1\le j\le l$, we have  
 \[
\sigma_{k,l}((i-1)l+j) = (j-1)k+i .
 \]
Algebras over the operad $\calP\otimes\calQ$ are called \emph{algebras with interchanging $\calP$- and $\calQ$-actions}.
\end{definition}

\smallskip 

The following rather obvious result on Boardman--Vogt tensor products is often useful. The closest reference for it that we could find is a particular case $\calP=\calQ$, see the work of Dunn~\cite[Prop.~1.6]{Dunn88}.

\begin{proposition}\label{prop:InterchangeGen}
Suppose that $\calP=\calT(\calX)/(\calR)$ and $\calQ=\calT(\calY)/(\calS)$ are presentations of the operads $\calP$ and $\calQ$ by generators and relations. Then 
 \[
\calP\otimes\calQ=\calT(\calX\sqcup\calY)/(\calR\sqcup\calS\sqcup\calI\calC),
 \]
where $\calI\calC$ are the relations $\gamma(x;y,\ldots,y)-\gamma(y;x,\ldots,x).\sigma_{k,l}$ where $x\in\calX(k)$ and $y\in\calY(l)$ are generators of $\calP$ and $\calQ$ respectively. In plain words, actions of two operads interchange if and only if the actions of their generators interchange.
\end{proposition}

\begin{proof}
It is sufficient to prove that if $p\in\calP(k)$ satisfies \eqref{eq:IC} with both $q\in\calQ(l)$, $q'\in\calQ(l')$, then $p.\alpha$ satisfies \eqref{eq:IC} with $q.\beta$  for all permutations $\alpha\in S_k$, $\beta\in S_l$, and $p$ satisfies \eqref{eq:IC} with $q\circ_s q'$ for any $1\le s\le l$.  Both of these are easily checked by direct inspection.
\end{proof}

\smallskip 

In the presence of constants, Boardman--Vogt tensor products exhibit various collapsing properties, which are variations of the Eckmann--Hilton argument~\cite{EH1961} in algebraic topology. Namely, the following result holds.

\begin{proposition}[{Fiedorowicz and Vogt~\cite[Prop.~3.8]{FW2015}}]
Suppose that the operads $\calP$ and $\calQ$ are such that $\calP(\underline{1})=\calQ(\underline{1})=\{\id\}$, and that the four components $\calP(\underline{0})$, $\calQ(\underline{0})$, $\calP(\underline{2})$, $\calQ(\underline{2})$ are nonempty. We have 
 \[ 
\calP\otimes\calQ\cong\uCom ,
 \]
where $\uCom$ is the operad of unital commutative associative algebras.
In particular, for the operad $\uAss$ of unital associative algebras, we have 
 \[
\uAss\otimes\uAss\cong\uCom . 
 \]
\end{proposition} 

Even in the set-up of this paper where constant operations are not allowed, unexpected phenomena arise. Let us consider the Boardman--Vogt square $\Ass\otimes\Ass$ of the operad $\Ass$ of non-unital associative algebras. It is generated by two associative products $\cdot$ and $\star$ satisfying the interchange law
 \[
(a_1\cdot a_2)\star (a_3\cdot a_4) = (a_1\star a_3)\cdot (a_2\star a_4) .
 \]
In \cite{Kock2007}, it was observed that an unexpected ``commutativity'' property holds in the operad $\Ass\otimes\Ass$. 

\begin{proposition}[{Kock~\cite[Prop.~2.3]{Kock2007}}]\label{prop:Kock}
In the Boardman--Vogt tensor product $\Ass\otimes\Ass$, the following holds in arity $16$:
\begin{align*}
&(a_1 \star a_2 \star a_3 \star a_4) \cdot (a_5 \star a_6 \star a_7 \star a_8) \cdot (a_9 \star a_{10} \star a_{11} \star a_{12}) \cdot (a_{13} \star a_{14} \star a_{15} \star a_{16}) = \\
&(a_1 \star a_2 \star a_3 \star a_4) \cdot (a_5 \star a_7 \star a_6 \star a_8) \cdot (a_9 \star a_{10} \star a_{11} \star a_{12}) \cdot (a_{13} \star a_{14} \star a_{15} \star a_{16}).
\end{align*}
In particular, the underlying $S_{16}$-module of $(\Ass\otimes\Ass)(\underline{16})$ is not free. 
\end{proposition}

The latter result was improved by the first author in his recent work with Madariaga~\cite{BM2016}.

\begin{proposition}[{\cite[Prop.~3.4 and Th.~4.2]{BM2016}}]\label{prop:MurraySara}
In each arity $n\le 8$ the underlying $S_n$-module of $(\Ass\otimes\Ass)(\underline{n})$ is free. The underlying $S_{9}$-module of $(\Ass\otimes\Ass)(\underline{9})$ is not free. In particular, the following relation implying that of Proposition~\ref{prop:Kock} holds:
 \[
(a_1 \star a_2) \cdot (a_3 \star a_4 \star a_5 \star a_6) \cdot (a_7 \star a_8 \star a_9) = 
(a_1 \star a_2) \cdot (a_3 \star a_5 \star a_4 \star a_6) \cdot (a_7 \star a_8 \star a_9) .
 \]
\end{proposition}

\begin{remark}
It is natural to ask what triggers the non-freeness of the underlying $S_{9}$-module of $(\Ass\otimes\Ass)(\underline{9})$. One natural guess which is suggested by the results of this paper is that $9=3\cdot 3$, where $3$ is the smallest arity in which the operad $\Ass$ has a nontrivial relation. It would be interesting to determine whether or not it is true that for operads $\calP$ and $\calQ$ whose underlying $\mathbb{S}$-modules are free, the Boardman--Vogt tensor product $\calP\otimes\calQ$ has a free underlying $\mathbb{S}$-module up to arity $kl-1$, where $k$ and $l$ are, respectively, the smallest arities where $\calP$ and $\calQ$ have relations. 
\end{remark}

\section{A geometric model for interchange of absolutely free operads}\label{sec:subdiv}

It turns out that interchanging $d$ absolutely free operads admits a remarkable geometric representation. We describe it in three steps. First, we consider a particular case when each of these operads is generated by one (not necessarily binary) generator, and define a map from the corresponding Boardman--Vogt tensor product into the operad of little $d$-rectangles. Next, we present a geometric construction of an arbitrary absolutely free operad in terms of subdivisions of the unit interval with some extra labelling data. Finally, we consider a certain superposition of these two constructions to represent arbitrary Boardman--Vogt tensor products. Note that at this stage we do not claim this representation to be faithful; the proof of its faithfulness is one of the key results of this paper which appears in Section~\ref{sec:proofs}.

\subsection{Interchanging one-generated absolutely free operads}\label{sec:OneGenerated}

Suppose that $\calT(\calX_1)$, $\calT(\calX_2)$, \ldots, $\calT(\calX_d)$ are (reduced connected) absolutely free operads, and 
suppose that for each collection $\calX_k$ there exists an integer $a_k>1$ for which 
 \[
\calX_k(\underline{a})=
\begin{cases}
S_{a_k}, \quad a=a_k,\\
\,\,\varnothing,\,\,\quad a\ne a_k ,
\end{cases}
 \]  
in other words, $\calX_k$ is freely generated by one element of arity $a_k$. 

Let us consider a version of the little $d$-cubes operad which we shall call the operad of little $d$-rectangles, and denote $\Rect_d$. By definition, its component of arity $n$ parametrises all possible ways to place $n$ rectangular boxes of dimension $d$ labelled $1,\ldots,n$ inside the unit cube so that their interiors are disjoint and their faces are parallel to the faces of the cube.  The operadic composition $\gamma(c;c_1,\ldots,c_m)$ of such configurations shrinks each of the configurations $c_i$ in the directions of the coordinate axes to ensure that the ambient unit cube fits exactly into the $i$-th rectangle of $c$, and then glues the configuration of rectangles thus obtained in place of that rectangle, adjusting the labels in the usual way.

\begin{definition}\label{def:CutOperad1}
For collections $\calX_1$, \ldots, $\calX_d$ as above,  the \emph{cut operad} $\calC^{(d)}_{\calX_\bullet}$ is the suboperad of $\Rect_d$ generated by the operations $\omega_k$, $1\le k\le d$, where $\omega_k$ is the configuration of rectangles 
\begin{gather*}
[0,1]^{k-1}\times [0,1/a_k]\times[0,1]^{d-k},\\
[0,1]^{k-1}\times [1/a_k,2/a_k]\times[0,1]^{d-k},\\
\vdots \\
[0,1]^{k-1}\times [(a_k-1)/a_k,1]\times[0,1]^{d-k},\\
\end{gather*}
numbered $1,\ldots,a_k$ in the order they are listed here. 
\end{definition}

Let us show that these operations interchange, which by Proposition \ref{prop:InterchangeGen} implies that there exists a surjective homomorphism $\calT(\calX_1)\otimes\calT(\calX_2)\otimes\cdots\otimes\calT(\calX_d)\to\calC^{(d)}_{\calX_\bullet}$. 

\begin{lemma}\label{lm:LittleRectInt}
The operations $\omega_i$ pairwise interchange. 
\end{lemma}

\begin{proof}
We see that the operation $\gamma(\omega_k;\omega_l,\ldots,\omega_l)$ is obtained by first cutting the unit cube into $a_k$ equal parts in the direction of the $k$-th coordinate hyperplane, and then cutting each of the parts thus obtained into $a_l$ equal parts in the directions of the $l$-th coordinate hyperplane. The operation $\gamma(\omega_l;\omega_k,\ldots,\omega_k)$ is obtained by first cutting the unit cube into $a_l$ equal parts in the direction of the $l$-th coordinate hyperplane, and then cutting each of the parts thus obtained into $a_k$ equal parts in the directions of the $k$-th coordinate hyperplane. The only difference between the two is the labelling of the interiors of the $a_ka_l$ parts thus obtained, and that difference is fixed by the permutation $\sigma_{k,l}$.
\end{proof}

\subsection{A geometric model for an absolutely free operad}

In this section, we present a geometric model for a (reduced connected) absolutely free operad.  Let us assume that $\calX$ is a free symmetric collection of finite sets with $\calX(0)=\calX(1)=\varnothing$. 

\begin{definition}\label{def:FreeOp}
An \emph{$\calX$-subdivision} of arity $n$ of a line segment $[a,b]\subset\mathbb{R}$ is defined by the following recursive rule:
\begin{itemize}[leftmargin=*]
\item The trivial subdivision consisting just of the segment $[a,b]$ without any extra data is the only $\calX$-subdivision of arity $1$. 
\item Choose an integer $m\ge 2$, a partition $n=n_1+\cdots+n_m$, and an element $w\in\calX(\underline{m})$. Let $t_k=((m-k)a+k b)/m$ with $k=1,\ldots,m-1$ be the $m-1$ points that divide $[a,b]$ into $m$ equal parts. Let us label each of these $m-1$ points by the element $w$, and impose arbitrary $\calX$-subdivisions of arities $n_1$, \ldots, $n_m$ on the $m$ segments 
 \[
[a,t_1], \quad [t_1,t_2], \quad  \ldots, \quad  [t_{m-1},b] . 
 \]
All $\calX$-subdivisions of arity $n$ are obtained in this way. 
\end{itemize} 
In other words, we cut the segment into several equal parts, cut each of the parts in several equal parts, etc., each time labelling the cuts by a generator of the free operad of appropriate arity.
\end{definition}

This definition trivially implies that $\calX$-subdivisions of arity $n$ of the unit interval $[0,1]$ are in one-to-one correspondence with elements of what is known as the absolutely free algebra (or term algebra) for the signature $\calX$. In order to model operads, we should label interiors of the segments into which we subdivide the unit interval by integers $1,\ldots,n$ in all possible ways. The operad composition comes from substitution of subdivisions in the same way as in Definition \ref{def:CutOperad1}; the only difference is that when inserting subdivisions we must also copy the labels of cuts. The operad thus obtained is immediately seen to be isomorphic to the absolutely free operad $\calT(\calX)$. 

\subsection{Interchanging several absolutely free operads}

We shall combine the previous two constructions to represent arbitrary tensor products of absolutely free operads. Let us now assume that $\calX_1$, \ldots, $\calX_d$ are free symmetric collections of finite sets with $\calX_k(\underline{0})=\calX_k(\underline{1})=\varnothing$ for all $1\le k\le d$.  

\begin{definition}
An \emph{$\calX_\bullet$-subdivision} of arity $n$ of a $d$-dimensional rectangle 
 \[
 R:=[a_1,b_1]\times \cdots\times [a_d,b_d]\subset\mathbb{R}^d
 \]
 is defined by the following recursive rule:
\begin{itemize}[leftmargin=*]
\item The trivial subdivision consisting just of the rectangle $R$ without any extra data is the only $\calX_\bullet$-subdivision of arity $1$. 
\item Choose an integer $1\le k\le d$, an integer $m\ge 2$, a partition $n=n_1+\cdots+n_m$, and an element $w\in\calX_k(\underline{m})$. Let $t_{k,l}=((m-l)a_k+l b_k)/m$, $1\le l\le m-1$, and let
 \[
\beta_l:=R\cap \{ x_k=t_{k,l} \}
 \]
be the $m-1$ hyperplane cuts orthogonal to the $k$-th direction that divide $R$ into $m$ equal parts. Let us label points of each of these cuts by the element $w$, and impose arbitrary $\calX_\bullet$-subdivisions of arities $n_1$, \ldots, $n_m$ on the $m$ parts 
\begin{gather*}
[a_1,b_1]\times \cdots\times[a_k,t_{k,1}],\times\cdots\times [a_d,b_d],\\
[a_1,b_1]\times \cdots\times[t_{k,1},t_{k,2}],\times\cdots\times [a_d,b_d],\\
\vdots \\
[a_1,b_1]\times \cdots\times[t_{k,m-1},b_k],\times\cdots\times [a_d,b_d] .
\end{gather*}
All $\calX_\bullet$-subdivisions of arity $n$ are obtained in this way. 
\end{itemize} 
In other words, we cut $R$ into several equal parts in one of the directions of coordinate hyperplanes, cut each of the parts in several equal parts, etc., each time labelling the cuts by a generator of appropriate arity.
\end{definition}

Let us use this geometric construction to define an operad. This generalises Definition \ref{def:CutOperad1}; the cut operad from that definition is tautologically isomorphic to the cut operad below when all operads $\calT(\calX_i)$ are one-generated.

\begin{definition}\label{def:CutOperad2}
The $d$-dimensional \emph{cut operad} $\calC^{(d)}_{\calX_\bullet}$ has, as its arity $n$ component, the $\calX_\bullet$-subdivisions of arity $n$ of the unit $d$-cube $[0,1]^d$ where interiors of the rectangles into which we subdivide the cube are labelled by integers $1,\ldots,n$ in all possible ways. 
The operad composition comes from substitution of labelled subdivisions in the same way as in the paragraph following Definition \ref{def:FreeOp}. 
\end{definition}

Let us establish that this construction gives a representation of the $d$-fold Boardman--Vogt tensor product  
$\calT(\calX_1)\otimes\cdots\otimes\calT(\calX_d)$.

\begin{proposition}
Let us consider, for each $x\in\calX_k(\underline{a_k})$, the operation $\omega_{k,x}\in\calC^{(d)}_{\calX_\bullet}(\underline{a_k})$ that corresponds to the $\calX_\bullet$-subdivision of the unit cube
\begin{gather*}
[0,1]^{k-1}\times[0,1/a_k]\times[0,1]^{d-k},\\
[0,1]^{k-1}\times[1/a_k,2/a_k]\times[0,1]^{d-k},\\
\ldots \\
[0,1]^{k-1}\times[(a_k-1)/a_k,1]\times[0,1]^{d-k},\\
\end{gather*}
where the parts are numbered $1,\ldots,a_k$ in the order they are listed here and all the $a_k-1$ cuts are labelled $x$.  The operations $\omega_{k,x}$ for various choices of $k$ and $x$ generate the operad $\calC^{(d)}_{\calX_\bullet}$. Moreover,  the operations $\omega_{k,x}$ and $\omega_{l,y}$ interchange for $k\ne l$. 
\end{proposition}

\begin{proof}
The first statement follows from the definition of the operad $\calC^{(d)}_{\calX_\bullet}$. The second one is proved completely analogously to Lemma \ref{lm:LittleRectInt}. 
\end{proof}

\begin{corollary}
There exists a surjective homomorphism 
 \[
\calT(\calX_1)\otimes\calT(\calX_2)\otimes\cdots\otimes\calT(\calX_d)\twoheadrightarrow\calC^{(d)}_{\calX_\bullet} .
 \]
\end{corollary}

\section{Proof of the main theorem}\label{sec:proofs}

In the previous section, we established that the cut operad $\calC^{(d)}_{\calX_\bullet}$ is a homomorphic image of $\calT(\calX_1)\otimes\cdots\otimes\calT(\calX_d)$. We shall now establish that these operads are isomorphic. The proof of this result is obtained through an indirect argument. To make that argument more transparent, we start sketch a proof of the recurrence relation for the numbers of elements in the cut operad representing $d$ interchanging binary operations, and then leave the combinatorics universe that was sufficient thus far and encode the more general recurrence relation homologically. The main result then follows from general properties of minimal resolutions of right modules over operads. 

\subsection{Sketch of enumeration of binary cuts in $d$ dimensions}

Counting binary cuts of the unit square is fairly straightforward.  Let $C^{(2)}_n$ be the number of distinct subdivisions of the unit square into $n$ pieces which are obtained by iterated bisections orthogonal to the coordinate axes.  Since there are two different directions, a first approximation to the recurrence relation is the same as for the Catalan numbers but with two different types of parentheses; namely,
  \[
  C^{(2)}_n = 2 \sum_{i=1}^{n-1} C^{(2)}_i C^{(2)}_{n-i},
  \]
as we need to choose the direction of the first cut, and then subdivide the two resulting rectangles. This involves double counting when we examine ``full'' bisections in two orthogonal directions corresponding to the interchange law. This double counting is easy to correct, and the actual recurrence relation is 
  \[
  C^{(2)}_n = 2 \sum_{i=1}^{n-1} C^{(2)}_i C^{(2)}_{n-i}\,\,\,  - \sum_{n_1+n_2+n_3+n_4=n} C^{(2)}_{n_1} C^{(2)}_{n_2} C^{(2)}_{n_3} C^{(2)}_{n_4} ,
  \]
which formalises the na\"ive idea that the doubly counted subdivisions are those where we make two perpendicular cuts, and then subdivide the four resulting squares. If we denote by $f_2(t)$ the generating function for the numbers $C^{(2)}_n$, this recurrence relation can be written in a concise form 
 \[
f_2(t)-2f_2(t)^2+f_2(t)^4=t ,
 \]
which takes into account the initial condition $C^{(2)}_1=1$. These numbers are documented in the OEIS entry ``association types in 2-dimensional algebra'' \cite[Sequence A236339]{OEIS}.  

This argument easily generalises to the $d$-dimensional case. The corresponding recurrence relation for the numbers $C^{(d)}_n$ of distinct subdivisions of the unit cube into $n$ parts becomes, by a similar inclusion-exclusion argument,
 \[
C^{(d)}_n 
  =
  \sum_{k=1}^d
  \,
  \Bigg[
  \,
  (-1)^{k-1}
  \,
  \binom{d}{k}
  \sum_{\substack{
  n_1,\ldots,n_{2^k}\ge 1 
  \\
  n_1+\cdots+n_{2^k}=n}}
  \prod_{i=1}^{2^k} \, C^{(d)}_{n_i}
  \,
  \Bigg] ,
 \]
or, in terms of the generating function $f_d(t)$ for the numbers $C^{(d)}_n$, 
 \[
\sum_{k=0}^d \binom{d}{k} f_d(t)^{2^k} = t . 
 \]
A rigorous proof of this relation follows from a more general result obtained by homological methods, see Corollary \ref{cor:FuncEq} below.

\subsection{A minimal resolution of the augmentation module} 

In this section, we give a homological statement which formalises the inclusion-exclusion argument above for the general cut operad. For that, we have to leave the set-theoretic context, and work with linearisations of the corresponding set operads. Below, the notation $\calC^{(d)}_{\calX_\bullet}$ is used for the linearised cut operad; we hope that it does not lead to a confusion.

\begin{lemma}\label{lm:CutResol}
There exists a minimal resolution 
 \[
\big((\mathbb{I}\oplus\k\calX_1)\Box\cdots\Box(\mathbb{I}\oplus\k\calX_d)\big)\circ\calC^{(d)}_{\calX_\bullet}
 \]
of the augmentation $\calC^{(d)}_{\calX_\bullet}$-module $\mathbb{I}$ by free right modules. Here the homological degree of all factors $\mathbb{I}$ is equal to zero, and the homological degree of $\k\calX_k$ is equal to $1$ for all $1\le k\le d$.
\end{lemma}

\begin{proof}
Let us denote, for brevity, 
\begin{equation}\label{eq:Homology}
\calH^{(d)}_{\calX_\bullet}=(\mathbb{I}\oplus\k\calX_1)\Box\cdots\Box(\mathbb{I}\oplus\k\calX_d) .
\end{equation}

We shall place the collection $\calH^{(d)}_{\calX_\bullet}$ in the same context as the $d$-dimensional cut operad. Namely, for each term $\k\calX_{i_1}\Box\cdots\Box\k\calX_{i_s}$ with $i_1<\cdots<i_s$ obtained by expanding the product \eqref{eq:Homology}, we choose a basis of elements
 \[
w_1\otimes\cdots\otimes w_s, \quad w_j\in\calX_{i_j}(\pi^{(j)}) ,
 \]
and associate with such element the $\calX$-subdivision of the unit cube into $n_1=|\pi^{(1)}|$ parts with hyperplanes parallel to $\{x_{i_1}=0\}$, then subdivision of each of the parts thus obtained into $n_2=|\pi^{(2)}|$ parts with hyperplanes parallel to $\{x_{i_2}=0\}$, etc. We label the $j$-th cut by $w_j$, and also label the interiors of the $d$-dimensional rectangles thus obtained using the orthogonal partitions: for $1\le m_1\le n_1$, \ldots $1\le m_s\le n_s$, the $(m_1,\ldots,m_s)$-rectangle obtains the label which is the only element of $\pi^{(1)}_{m_1}\cap\cdots\cap\pi^{(s)}_{m_s}$.  

The collection $\calH^{(d)}_{\calX_\bullet}\circ\calC^{(d)}_{\calX_\bullet}$ can now be viewed as follows. Its basis elements are indexed by ${\calX_\bullet}$-subdivisions of the unit cube, where we take
 a ``full'' subdivision from $\calH^{(d)}_{\calX_\bullet}$, and then insert inside each of its boxes a ${\calX_\bullet}$-subdivision of the unit $d$-cube. To make the distinction between two types of cuts clear, we shall refer to cuts coming from $\calH^{(d)}_{\calX_\bullet}$ as black cuts, and the cuts coming from $\calC^{(d)}_{\calX_\bullet}$ as white cuts, so that the ${\calX_\bullet}$-subdivisions we use are now two-coloured. 

Suppose that $c$ is a basis element of $\calH^{(d)}_{\calX_\bullet}\circ\calC^{(d)}_{\calX_\bullet}$. We shall call, for $i=1,\ldots,k$, the hyperplane $\alpha_i=\{x_i=0\}$ a \emph{cut-through direction} for $c$ if there exists an integer $n_i\ge 2$ and an element $v\in\calX_i(\underline{n_i})$ for which the hyperplane pieces parallel to $\alpha_i$  which cut the unit cube into $n_i$ equal parts are fully covered by cuts of $c$, and all the points of those cuts are labelled by the element $v$.

We now define a structure of a chain complex on $\calH^{(d)}_{\calX_\bullet}\circ\calC^{(d)}_{\calX_\bullet}$. For that, it is convenient to assign to a two-coloured ${\calX_\bullet}$-subdivision $c$ a basis element  
 \[
C=(x_{v_1,i_1}x_{v_2,i_2}\cdot\cdots\cdot x_{v_r,i_r}\otimes \xi_{w_1,j_1}\wedge\cdots\wedge \xi_{w_s,j_s})c
 \]
of $\calH^{(d)}_{\calX_\bullet}\circ\calC^{(k)}_{\calX_\bullet}$, where $\alpha_{i_1}$, \ldots, $\alpha_{i_r}$ are the cut-through directions for $c$ with the respective labels $v_1$, \ldots, $v_r$, and $\alpha_{j_1}$, \ldots, $\alpha_{j_s}$ are the black cuts of $c$ with the respective labels $w_1$, \ldots, $w_s$. Here $x_{v,i}$, $v\in\calX_i$, are formal commuting variables, and $\xi_{w,j}$, $w\in\calX_j$, are formal anti-commuting variables. 

We define a linear map $\mathbf{d}$ of homological degree $-1$ on $\calH^{(d)}_{\calX_\bullet}\circ\calC^{(d)}_{\calX_\bullet}$ as follows. For a basis element $C$
as above, we put 
 \[
\mathbf{d}(C)=\sum_{p=1}^s (-1)^{p-1} (x_{v_1,i_1}x_{v_2,i_2}\cdot\cdots\cdot x_{v_r,i_r}x_{w_p,j_p}\otimes \xi_{w_1,j_1}\wedge\cdots\wedge \hat{\xi}_{w_p,j_p}\wedge\cdots\wedge \xi_{w_s,j_s})c^{(p)} ,
 \]
where $c^{(p)}$ is the ${\calX_\bullet}$-subdivision for which the colour of the black cuts in the direction of the hyperplane $\alpha_{j_p}$ is changed from black to white. By a direct computation, $d^2=0$, so $\calH^{(d)}_{\calX_\bullet}\circ\calC^{(d)}_{\calX_\bullet}$ acquires a chain complex structure. 

We also define a linear map $\mathbf{h}$ of homological degree $1$ on $\calH^{(d)}_{\calX_\bullet}\circ\calC^{(d)}_{\calX_\bullet}$ as follows. For a basis element $C$ as above, we put 
 \[
\mathbf{h}(C)=\sum_{q=1}^r (x_{v_1,i_1}x_{v_2,i_2}\cdot\cdots\cdot \hat{x}_{v_q,i_q}\cdot\cdots\cdot x_{v_r,i_r}\otimes \xi_{v_q,i_q}\wedge\xi_{w_1,j_1}\wedge\cdots\wedge \xi_{w_s,j_s})c_{(q)} ,
 \]
where $c_{(q)}$ is the ${\calX_\bullet}$-subdivision for which the colour of the hyperplance in the $q$-th cut-through direction for $c$ is changed from white to black. 
By a direct computation,
 \[
(\mathbf{d}\mathbf{h}+\mathbf{h}\mathbf{d})(C)=(n_b(c)+n_w(c))C , 
 \]
where $n_b(c)$ is the number of the black cuts in $c$ and $n_w(c)$ is the number of cut-through directions for $c$; in fact, the formulas for the differential and the map $\mathbf{h}$ are designed in such a way that they mimic the classical Koszul complex (the polynomial de Rham complex). Note that for every $p$ the subcollection $\left(\calH^{(d)}_{\calX_\bullet}\circ\calC^{(d)}_{\calX_\bullet}\right)_p$ spanned by all basis elements for which $n_b(c)+n_w(c)=p$ is closed under both $\mathbf{d}$ and $\mathbf{h}$. For $n>0$, let us define a map $\mathbf{h}'$ on $\left(\calH^{(d)}_{\calX_\bullet}\circ\calC^{(d)}_{\calX_\bullet}\right)_p$ by the formula $\mathbf{h}'=\frac1p \mathbf{h}$. Clearly, $\mathbf{d}\mathbf{h}'+\mathbf{h}'\mathbf{d}=\id$, and hence the chain complex $\left(\calH^{(d)}_{\calX_\bullet}\circ\calC^{(d)}_{\calX_\bullet}\right)_p$ is acyclic. Also, we have $\left(\calH^{(d)}_{\calX_\bullet}\circ\calC^{(d)}_{\calX_\bullet}\right)_0\cong\mathbb{I}$, as for all non-unary elements there is either at least one black cut, or at least one cut-through direction (or both). 

Finally, it is obvious that this resolution is minimal, as the differential creates at least one white cut, thus landing in the augmentation ideal.
\end{proof}

\subsection{Faithfulness of the combinatorial representation of interchange}

We are finally able to establish that the cut operad represents the Boardman--Vogt tensor product faithfully. 

\begin{theorem}\label{th:Iso}
We have 
 \[
\calC^{(d)}_{\calX_\bullet}\cong \calT(\calX_1)\otimes\cdots\otimes\calT(\calX_d) .
 \]
\end{theorem}

\begin{proof}
Let us move to the linear context, and replace the set operads $\calC^{(d)}_{\calX_\bullet}$ and $\calT(\calX_1)\otimes\cdots\otimes\calT(\calX_d)$ by their linearisations (keeping the same notation). From Lemma \ref{lm:CutResol}, we know that 
 \[
\left((\mathbb{I}\oplus\calX_1)\Box\cdots\Box(\mathbb{I}\oplus\calX_d)\right)\circ\calC^{(d)}_{\calX_\bullet}
 \]
is a minimal resolution of $\mathbb{I}$, the augmentation module for $\calC^{(d)}_{\calX_\bullet}$ by free right $\calC^{(d)}_{\calX_\bullet}$-modules.  It is well known that minimal resolutions of modules are defined uniquely up to an isomorphism, and that for a reduced connected $\k$-linear operad $\calO$ the generators of a minimal resolution of the augmentation $\calO$-module $\mathbb{I}$ in low homological degrees have easy interpretations in terms of that operad: the generators of homological degree $1$ correspond to the minimal set of generators $\calY$ for $\calO$ and the generators of homological degree $2$ correspond to the minimal set of relations (minimal set of generators of the kernel of the surjection $\calT(\calY)\twoheadrightarrow\calO$). In our case, elements of homological degree $1$ are indexed by choices of direction of [simultaneous] black cuts, and a label for such cut, which is not surprising: as we know, the operad $\calC^{(d)}_{\calX_\bullet}$ is generated by $\calX_\bullet$. Elements of homological degree $2$ are indexed by choices of two directions of [simultaneous] black cuts and their labels, say $p$ and $q$. The differential of such an element is the difference of two elements where the  simultaneous black cuts in one of the two directions are made white. Such an element encodes a relation in the operad: its differential is the difference of two equal elements where all the black cuts are made white; thus such an element represents the corresponding interchange law between $p$ and $q$. Thus, all relations of  $\calC^{(d)}_{\calX_\bullet}$ follow from interchange laws between the generating operations. We now refer to the presentations of Boardman--Vogt tensor products given by Proposition~\ref{prop:InterchangeGen} to complete the proof.
\end{proof}

The following result shows that, by contrast with Propositions \ref{prop:Kock} and \ref{prop:MurraySara}, no unexpected symmetries arise for interchanging absolutely free structures. 

\begin{corollary}
The underlying $S_n$-module of $(\calT(\calX_1)\otimes\cdots\otimes\calT(\calX_d))(n)$ is free. 
\end{corollary}

\begin{proof}
This follows from the trivial observation that the underlying $S_n$-module of $\calC^{(d)}_{\calX_\bullet}(n)$ is free, since all possible labelling of rectangles are allowed.
\end{proof}

\begin{remark}
Let us mention an application of Theorem \ref{th:Iso} to a more ``classical'' question stated in terms of varieties of algebras. It follows immediately from that theorem 
that in the variety of nonassociative algebras defined by $d$ binary operations
$\star_1, \dots, \star_d$ with no symmetry satisfying the $d(d-1)/2$ interchange laws,
any free algebra generated by a set $X$ has a "monomial basis" consisting of all subdivisions of the unit $d$-cube into 
smaller $d$-rectangles with disjoint interiors by iterated bisections orthogonal to coordinate axes with additional labelling:
each of those $d$-rectangles should be given a label from $X$.  The multiplication of these ``labelled subdivisions'' may then
be defined geometrically as follows:  If $p$ and $q$ are labelled subdivisions, then for
$1 \le i \le d$, the product $p \star_i q$ is the labelled subdivision $\tfrac12 \big( p \cup ( e_i + q ) \big)$,
where $e_i$ is the unit vector in the $i$th direction.
\end{remark}

Combining all the results we proved, we can now establish the key conceptual result of this paper.

\begin{theorem}
\label{maintheorem}
Let $\calT(\calX_1)$, \ldots, $\calT(\calX_d)$ be reduced connected absolutely free set operads. There exists a minimal resolution 
 \[
\big((\mathbb{I}\oplus\k\calX_1)\Box\cdots\Box(\mathbb{I}\oplus\k\calX_d)\big)\circ\big(\calT(\calX_1)\otimes\cdots\otimes\calT(\calX_d)\big)
 \]
of the augmentation module $\mathbb{I}$ over (the linearised version of) the $d$-fold Boardman--Vogt tensor product $\calT(\calX_1)\otimes\cdots\otimes\calT(\calX_d)$ by free right modules. Here the homological degree of all factors $\mathbb{I}$ is equal to zero, and the homological degree of $\calX_k$ is equal to $1$ for all $1\le k\le d$.
\end{theorem}

\begin{proof}
By Theorem \ref{th:Iso}, we have $\calC^{(d)}_{\calX_\bullet}\cong \calT(\calX_1)\otimes\cdots\otimes\calT(\calX_d)$, so the operad $\calC^{(d)}_{\calX_\bullet}$ in Lemma \ref{lm:CutResol} can be replaced by 
 \[
\calT(\calX_1)\otimes\cdots\otimes\calT(\calX_d) ,
 \]
the $d$-fold Boardman--Vogt tensor product. 
\end{proof}

For completeness, we state the general version of the inclusion-exclusion functional equation discussed in the introduction. To that end,  we shall need the linear map $N$ from the algebra of Dirichlet series to the algebra of formal power series for which $N(n^{-s})=x^n$.

\begin{corollary}\label{cor:FuncEq}
Let $D^{(d)}_{{\calX_\bullet}}(s)$ be the Dirichlet generating function of Euler characteristics of the collection 
 \[
(\mathbb{I}\oplus\calX_1)\Box\cdots\Box(\mathbb{I}\oplus\calX_d) 
 \]
with the homological grading as described in Lemma \ref{lm:CutResol} and Theorem \ref{maintheorem}.
We have
 \[
D^{(d)}_{{\calX_\bullet}}(s)=\prod_{k=1}^d\left(1-\sum_{n\ge 2}\frac{\dim\calX_k(\underline{n})}{n! n^s}\right) .
 \]
Furthermore, the power series  
 \[
g^{(d)}_{\calX}(x)=N(D^{(d)}_{{\calX_\bullet}}(s)) 
 \]
is the compositional inverse of the generating function 
 \[
f^{(d)}_{\calX_\bullet}(x)=\sum_{n\ge 1}\frac{|\calT(\calX_1)\otimes\cdots\otimes\calT(\calX_d)(\underline{n})|}{n!}x^n .
 \]
\end{corollary}

\begin{proof}
The first statement follows from the fact that the operation $\Box$ categorifies the product of Dirichlet series. The second statement is obtained by computing the Euler characteristics of the complex
 \[
\left((\mathbb{I}\oplus\calX_1)\Box\cdots\Box(\mathbb{I}\oplus\calX_d)\right)\circ\left(\calT(\calX_1)\otimes\cdots\otimes\calT(\calX_d)\right)
 \]  
in two different ways, directly via the composition of collections and via the homology (which is $\mathbb{I}$ in degree zero and arity one, and zero otherwise). 
\end{proof}

\section{Concluding remarks}\label{sec:conclusion}

The statement of Theorem \ref{maintheorem} is aesthetically appealing, and, if we note that
 \[
\big(\calX\circ\calT(\calX)\to\calT(\calX)\big)\cong \big(\mathbb{I}\oplus\k\calX\big)\circ\calT(\calX)
 \]
is the minimal resolution of the augmentation $\calT(\calX)$-module $\mathbb{I}$ by free right modules, admits the following obvious generalisation to arbitrary set operads.

\begin{definition}
Let $\calP_1$, \ldots, $\calP_d$ be reduced connected set operads, and suppose that $\calV_k\circ\calP_k$ is the underlying $\mathbb{S}$-module of the minimal resolution of the augmentation module $\mathbb{I}$ over (the linearised version of) $\calP_k$ by right modules. 
We say that this $d$-tuple of operads \emph{has $\Box$-multiplicative homology} if there exists a minimal resolution 
 \[
(\calV_1\Box\cdots\Box\calV_d)\circ(\calP\otimes\cdots\otimes\calP_d)
 \]
of the augmentation module $\mathbb{I}$ over (the linearised version of) the $d$-fold Boardman--Vogt tensor product $\calP_1\otimes\cdots\otimes\calP_d$ by free right modules. 
\end{definition}

In this section, we discuss intuition behind this property, present two examples showing that it should not be expected to hold in general, and make a conjecture generalising Theorem \ref{maintheorem} to a slightly wider class of examples. 

\smallskip 

Let us first offer some intuition behind $\Box$-multiplicativity. For that, let us consider the right module version of the Boardman--Vogt tensor product $\tilde{\otimes}$ (obtained from the bimodule tensor product introduced by Dwyer and Hess \cite{DH2014}), which satisfies the following two properties crucial for us.

\begin{proposition}\label{prop:DHK}
Assume that we consider operads and modules in simplicial sets. 

\noindent 
1 (Dwyer and Hess \cite[Th.~1.14]{DH2014}). For any $\mathbb{S}$-modules $\calV_1$, \ldots, $\calV_d$ we have
 \[
(\calV_1\circ\calP_1)\tilde{\otimes}\cdots\tilde{\otimes}(\calV_d\circ\calP_d)\cong(\calV_1\Box\cdots\Box\calV_d)\circ(\calP\otimes\cdots\otimes\calP_d) . 
 \]

\noindent 
2 (Dwyer, Hess, and Knudsen \cite{Hess2017}). Let $\calP$ and $\calQ$ be operads, and let $\calG$ be a right $\calQ$-module. If $\calG$ is cofibrant in the projective model structure, then the functor $-\otimes\calG$ (from right $\calP$-modules to right $\calP\otimes\calQ$-modules) is a left Quillen functor.
\end{proposition}

Let us try to proceed, for the sake of the argument, as if these results were available in the $\k$-linear context. We consider, for each $1\le k\le d$, the dg module $\calV_k\circ\calP_k$ which is the minimal resolution of the augmentation $\calP_k$-module $\mathbb{I}$ by free right modules.  By a result of Fresse \cite[Prop.~14.2.2]{Fresse2009}, a minimal resolution is cofibrant whenever $\calV_k$ and $\calP_k$ are cofibrant as $\mathbb{S}$-modules. Thus, under this extra assumption it would follow from the left Quillen property that 
 \[
(\calV_1\circ\calP_1)\tilde{\otimes}\cdots\tilde{\otimes}(\calV_d\circ\calP_d)\cong(\calV_1\Box\cdots\Box\calV_d)\circ(\calP\otimes\cdots\otimes\calP_d) 
 \]
is quasi-isomorphic to $\mathbb{I}$, so the $d$-tuple of operads $\calP_1$, \ldots, $\calP_d$ have $\Box$-multiplicative homology.  

\smallskip 

There is, however, a big problem with this argument (and hence it is only good as an intuitive explanation of $\Box$-multiplicativity): Proposition \ref{prop:DHK} is not available in the linear setting, and there is nothing on the level of simplicial sets for us to linearise: for operads in simplicial sets there is no notion of augmentation. In fact, the following example shows that cofibrancy as $\mathbb{S}$-modules is certainly not enough.

\begin{example}\label{ex:AssAss}
Consider the symmetric operad $\Ass$ of non-unital associative algebras. Note that the underlying $\mathbb{S}$-module of $\Ass$ is free, and that we have the Koszul (minimal) resolution $\Ass^{\ac}\circ\Ass$ of the augmentation module. However, it is clear that the minimal resolution of the augmentation module for $\Ass\otimes\Ass$ cannot be of the form 
 \[
\big(\Ass^{\ac}\Box\Ass^{\ac}\big)\circ\big(\Ass\otimes\Ass\big) , 
 \]
as computing Euler characteristics would have immediately implied freeness of the underlying $\mathbb{S}$-module of $\Ass\otimes\Ass$, contradicting Propositions \ref{prop:Kock} and \ref{prop:MurraySara}. Thus, $\Box$-multiplicativity of homology fails in this case. 
\end{example}

The following example of failure of $\Box$-multiplicativity for homology is less surprising, since the corresponding operads are not $\Sigma$-cofibrant on the level of sets. 

\begin{example}\label{ex:FreeFree}
Let us take $\calT(\calX_1)\cong\calT(\calX_2)$ to be the free operad generated by one \emph{commutative} binary operation. We have the minimal resolutions 
 \[
\big(\mathbb{I}\oplus\k\calX_1\big)\circ\calT(\calX_1) \quad \text{ and } \quad \big(\mathbb{I}\oplus\k\calX_2\big)\circ\calT(\calX_2)
 \]
for the respective augmentation modules, but the minimal resolution of the augmentation module for $\calT(\calX_1)\otimes\calT(\calX_2)$ cannot be of the form
 \[
\big((\mathbb{I}\oplus\calX_1)\Box(\mathbb{I}\oplus\calX_2)\big)\circ\big(\calT(\calX_1)\otimes\calT(\calX_2)\big) ,
 \]
since by a direct computation the space $\big((\mathbb{I}\oplus\calX_1)\Box(\mathbb{I}\oplus\calX_2)\big)(\underline{4})$ is six-dimensional, and the space of generators of the minimal resolution in arity $4$ is five-dimensional. Thus, $\Box$-multiplicativity of homology fails in this case as well. 
\end{example}

We conclude with a conjecture that slightly strengthens Theorem \ref{maintheorem}.

\begin{conjecture}\label{conj:TensorWithFree}
Suppose that the reduced connected set operad $\calO$ is free as an $\mathbb{S}$-module, and that $\calF=\calT(\calX)$ is a reduced connected absolutely free set operad.
The pair of operads $\calO$ and $\calF$ has $\Box$-multiplicative homology.
\end{conjecture}

\end{document}